\documentclass[a4paper,10pt,reqno]{amsart}

\usepackage{amsmath,amssymb}
\usepackage{url}
\usepackage[all]{xy}
\usepackage{longtable}
\usepackage{supertabular}
\usepackage{multirow}

\newcommand{\Q}[1]{\mathbb{Q}(\sqrt{#1})}
\newcommand{\Z}{\mathbb{Z}}
\newcommand{\N}{\mathbb{N}}

\newcommand{\Fro}{\mathfrak{o}}

\newcommand{\Fra}{\mathfrak{a}}

\newcommand{\Frn}{\mathfrak{n}}
\newcommand{\Frs}{\mathfrak{s}}
\newcommand{\Frv}{\mathfrak{v}}

\newtheorem{Rmk}{Remark}
\newtheorem{Lem}{Lemma}
\newtheorem{Thm}{Theorem}

\newcommand{\qf}[1]{\langle #1 \rangle}
\newcommand{\conj}[1]{\overline{#1}}

\newcommand{\binlattice}[4]{\begin{small}\begin{pmatrix}%
  #1 & #2 \\
  #3 & #4
\end{pmatrix}\end{small}}
\newcommand{\terlattice}[9]{\begin{small}{\begin{pmatrix}%
  #1 & #2 & #3 \\
  #4 & #5 & #6 \\
  #7 & #8 & #9
\end{pmatrix}}\end{small}}

\newcommand{\comega}{\conj\omega}

\newcommand{\Tr}{\operatorname{Tr}}

\newcommand{\rank}{\operatorname{rank}}

\newcommand{\ra}{\rightarrow}

\newcommand{\vol}[1]{{\text{\rm vol:}#1}}

\begin{document}

\title[Even universal binary Hermitian lattices]
{Even universal binary Hermitian lattices \\over imaginary quadratic fields}

\subjclass[2000]{Primary 11E39; Secondary 11E20, 11E41}

\author[Byeong Moon Kim]{Byeong Moon Kim}
\address{Department of Mathematics, Kangnung National University, Kangnung, 210-702, Korea.}%
\email{kbm@kangnung.ac.kr}

\author[Ji Young Kim]{Ji Young Kim}
\address{School of Mathematics, Korea Institute for Advanced Study, Hoegiro 87, Dongdaemun-gu, Seoul, 130-722, Korea.}%
\email{jykim@kias.re.kr}

\author[Poo-Sung Park]{Poo-Sung Park}
\address{Department of Mathematics Education, Kyungnam University, 449 Wolyong-Dong, Masan, Kyungnam, 631-701, Korea.}%
\email{sung@kias.re.kr}

\begin{abstract}
A positive definite even Hermitian lattice is called \emph{even universal} if it represents all even positive integers. We introduce a method to get all even universal binary Hermitian lattices over imaginary quadratic fields $\Q{-m}$ for all positive square-free integers $m$ and we list optimal criterions on even universality of Hermitian lattices over $\Q{-m}$ which admits even universal binary Hermitian lattices.
\end{abstract}

\keywords{binary Hermitian lattices, even universal}
\maketitle%

\section{Introduction}

Lagrange's four square theorem says that every positive integer can be written as a sum of four
squares of integers. A quadratic form, like sum of four squares, is called \emph{universal}, if it
represents all positive integers. In 1997, Conway, Schneeberger and Bhargava announced the \emph{fifteen theorem} for classical universal quadratic forms, which characterizes the universality by representability of nine critical numbers, namely, $1$, $2$, $3$, $5$, $6$, $7$, $10$, $14$ and $15$ (see \cite{jhC-00}, \cite{mB-00}). Recently, Bhargava and Hanke enunciated that they proved the \emph{290-theorem} characterizing the universality of (nonclassical) quadratic forms. That is, if a (nonclassical) quadratic form represents 29 numbers, say, $1$, $2$, $3$, $5$, $6$, $7$, $10$, $13$, $14$, $15$, $17$, $19$, $21$, $22$, $23$, $26$, $29$, $30$, $31$, $34$, $35$, $37$, $42$, $58$, $93$, $110$, $145$, $203$, $290$, then it is universal (see \cite{mB-jH-06}).

It is a natural attempt to extend these results to Hermitian lattices over imaginary quadratic
fields. A positive definite Hermitian lattice over an imaginary quadratic field is called \emph{universal} if it represents all positive integers. There are many contributions to list all universal binary Hermitian lattices over imaginary quadratic fields (see \cite{agE-aK-97(1)}, \cite{hI-00}, \cite{jhK-psP-07}). In the previous article \cite{bmK-jyK-psP(1)}, the authors provided a list of universal Hermitian lattices over imaginary quadratic field $\Q{-m}$ for all positive square-free integers $m$, and the  Conway-Schneeberger-Bhargava type criterion on universality, so called the \emph{fifteen theorem for universal Hermitian lattices}.

The goal of this article is to prove the analogue of the previous paper \cite{bmK-jyK-psP(1)} for even universal Hermitian lattices over imaginary quadratic fields $\Q{-m}$ for all possible positive square-free integers $m$. First, we find all imaginary quadratic fields $\Q{-m}$ for positive square-free integers $m$ that admit an even universal binary Hermitian lattice and classify all such lattices. A Hermitian lattice is \emph{even} if its norm ideal is generated by $2$. A primitive even Hermitian lattice is called \emph{even universal} if it represents all positive even integers. Next, we determine the \emph{even critical set} $C[m]$ of an imaginary quadratic field $\Q{-m}$ which admits an even universal binary Hermitian lattice.

\section{Review on Hermitian lattices over imaginary quadratic fields}

In this chapter, we will review Hermitian lattices, in particular, we will review a matrix presentation of both free and non-free Hermitian lattices. That was suggested by the authors in \cite{bmK-jyK-psP(1)}.

Let $E$ denote an imaginary quadratic field $\Q{-m}$ for a positive square-free integer $m$ with nontrivial $\mathbb{Q}$-involution and let $\Fro$ be the ring of integers of $E$. It is well-known that $\Fro = \Z[\omega]$, where $\omega = \sqrt{-m}$ if $m \equiv 1, 2 \pmod{4}$ or $\omega = \frac{1+\sqrt{-m}}2$ if $m \equiv 3 \pmod{4}$. An \emph{$\Fro$-lattice} $L$ means a finitely generated $\Fro$-module in the Hermitian space $(V,H)$, where $V$ is an $n$-dimensional vector space over $E$ with a nondegenerate Hermitian map $H: V \times V \ra E$. Without specific mention, we will assume that all lattices are positive definite and $H(v_1, v_2) \in \Fro$ for all $v_1, v_2 \in L$. Since $H(v_1,v_2)=\conj{H(v_2,v_1)}$, the (Hermitian) norm $H(v) = H(v,v)$ is in $\mathbb{Z}$ for every $v \in L$. If $a = H(v)$ for some $v \in L$, we say that $a$ is represented by $L$.

The lattice $L$ can be written as
\[
    L = \Fra_1 v_1 + \Fra_2 v_2 + \dotsb + \Fra_n v_n
\]
with ideals $\Fra_i \subset \Fro$ and vectors $v_i \in V$. The
\emph{volume ideal} $\Frv{L}$ of $L$ is defined by
\[
    \Frv{L} = (\Fra_1 \conj{\Fra}_1) (\Fra_2 \conj{\Fra}_2) \dotsb (\Fra_n\conj{\Fra}_n) \det(H(v_i,v_j))_{1\le i,j \le n}.
\]
The volume of $L$ is a principal ideal. The \emph{norm ideal} $\Frn{L}$ of $L$ is an $\Fro$-ideal generated by the set $\{H(v)| v \in L\}$. The \emph{scale ideal} $\Frs{L}$ of $L$ is an $\Fro$-ideal generated by the set $\{H(v,w)| v,w \in L\}$. It is clear that $\Frn{L} \subseteq \Frs{L}$. If $\Frn{L} = \Frs{L}$, then we call $L$ \emph{normal}. Otherwise, we call $L$ \emph{subnormal}. This paper focuses on the subnormal lattices with $\Frn{L} = 2\Fro$.

If $L$ is a free $\Fro$-module, then we can write $L = \Fro v_1 + \dotsb + \Fro v_n$. The matrix $M_L = (H(v_i,v_j))_{1\le i,j \le n}$ is called the Gram matrix of $L$ and is a matrix presentation of $L$. If the matrix is diagonal, we denote it by $\qf{H(v_1),H(v_2),\dotsc,H(v_n)}$. But, if $L$ is not a free $\Fro$-module, then $L= \Fro v_1 + \dotsb + \Fro v_{n-1} + \Fra v_n$ for some ideal $\Fra \subset \Fro$ \cite[81:5]{otO-73}. Since any ideal in $\Fro$ is generated by at most two elements, we can write $L= \Fro v_1 + \dotsb + \Fro v_{n-1} + (\alpha, \beta)\Fro v_n$ for some $\alpha, \beta \in \Fro$. Therefore, we consider the following $(n+1) \times (n+1)$-matrix as a formal Gram matrix for $L$:
\[
    M_L =
    \begin{pmatrix}
        H(v_1, v_1)         & \dotsc & H(v_1, \alpha v_n)        & H(v_1, \beta v_n)\\
        \vdots              & \ddots & \vdots                    &\vdots\\
        H(\alpha v_n,  v_1) & \dotsc & H(\alpha v_n, \alpha v_n) & H(\alpha v_n, \beta v_n)\\
        H(\beta  v_n,  v_1) & \dotsc & H(\beta  v_n, \alpha  v_n) & H(\beta  v_n, \beta v_n)\\
    \end{pmatrix}.
\]
Note that this matrix $M_L$ is positive semi-definite, but it represents an $n$-ary positive definite Hermitian lattice. A scaled lattice $L^a$ is obtained from the Hermitian map $H_{L^a}(\cdot,\cdot) = a H_L(\cdot,\cdot)$ with $a \in \mathbb{Q}$. If $M$ is a matrix presentation of a lattice $L$, $aM$ is the matrix presentation of a scaled lattice $L^a$.

Considering a $2n$-dimensional vector space $\widetilde{V}$ over $\mathbb{Q}$ corresponding to $V$ as defined in \cite{nJ-40}, we can regard $(V,H)$ over $E$ as a $2n$-dimensional quadratic space
$(\widetilde{V}, B_H)$ over $\mathbb{Q}$ where $B_H (x,y) = \frac{1}{2} [ H(x,y) + H(y,x) ] = \frac{1}{2} \Tr_{E/\mathbb{Q}}(H(x,y))$. Analogously, by viewing $L$ as a $\mathbb{Z}$-lattice in $(\widetilde{V}, B_H)$ we can obtain a quadratic $\mathbb{Z}$-lattice $\widetilde{L}$ in $\widetilde{V}$ associated to a Hermitian $\Fro$-lattice $L$ in $V$. For example, the quadratic form associated to the Hermitian lattice $\qf{1}$ over $\Q{-m}$ is
\[
    \begin{cases}
        ~x_1^2 + m y_1^2                       &\text{ if } m \equiv 1, 2 \pmod{4},\\
        ~x_1^2 + x_1 y_1 + \frac{1+m}{4} y_1^2 &\text{ if } m \equiv 3
        \pmod{4}.
    \end{cases}
\]

For unexplained terminologies, notations and basic facts about quadratic forms and Hermitian
lattices, we refer the readers to \cite{otO-73}, \cite{nJ-40} and \cite{bmK-jyK-psP(1)}.

\section{Even universal binary Hermitian lattices}

A (even resp.) Hermitian lattice is said to be (\emph{even} resp.) \emph{universal} if it represents every positive (even resp.) integers. If a Hermitian lattice $L$ is not (even resp.) universal, we define the (\emph{even} resp.) \emph{truant} of $L$ to be the smallest positive (even resp.) integer not represented by $L$. An \emph{even escalation} of a nonuniversal Hermitian lattice $L$ is defined to be a process of getting any Hermitian lattice which generated by $L$ and any vectors whose Hermitian norm is equal to the even truant of $L$. An \emph{even escalation lattice} is a lattice that can be obtained from a finite sequence of successive even escalations starting from the trivial lattice $\{ 0 \}$. Then $\qf{2}$ is the even escalation lattice of rank $1$. A \emph{spurious escalation} of a Hermitian lattice $L$ is a process of finding any lattice $\widetilde{L}$ which contains the lattice $L$ and whose rank is the same as $\rank{L}$. Since $\Frv L \subseteq \Frv \widetilde{L}$, a spurious escalation terminates with finite steps. Important in the proof Theorem \ref{Thm:EvenUniv} is the notion of even escalations and spurious escalations. In particular, to find all candidates of even universal binary Hermitian lattices, the spurious escalations are imperative. In this article, if there is no confusion, we use the term escalation instead of the even escalation.

We set
\[
    [a_1, \alpha,  a_2] := \binlattice{a_1}{\alpha}{\conj\alpha}{a_2} ~\text{ and }~
    [a_1, a_2, a_3 \,; \alpha, \beta, \gamma]
                        := \terlattice{a_1}{\alpha}{\beta}{\conj\alpha}{a_2}{\gamma}{\conj\beta}{\conj\gamma}{a_3}\\
\]
for convenience.

\begin{Lem}\label{Lem:EvenUniv}
A primitive even Hermitian lattice exists over the field $\Q{-m}$ only when $m \not\equiv 3 \pmod{4}$.
\end{Lem}

\begin{proof}
Suppose $L$ is a primitive even universal binary Hermitian lattice over $\Q{-m}$, where $m \equiv 3 \pmod{4}$. Let $u$ and $v$ be vectors of $L$ such that $H(u) = 2a$, $H(v) = 2b$, $H(u, v) = c +
d\omega$ for some integers $a,b,c,d$. Since $L$ is primitive, both $c$ and $d$ are not even. If $d$ is odd, then $H(u + v)= 2a + 2b + 2c +d$ is not even. If $d$ is even, then $c$ is odd and $H(u +
\omega v) = 2a + \frac{1+m}{2} b + c + \frac{1+m}{2} d$ is not even. Hence we get the result.
\end{proof}

\begin{Rmk}\label{Rmk:BH}
Let $L$ be an even binary Hermitian lattice over $\Q{-m}$ and let $f_L$ be a corresponding integral quadratic form of $L$. Since $\frac{1}{2} f_L$ is a nonclassical integral quadratic form, the
universality of $\frac{1}{2} f_L$ can be determined by Bhargava and Hanke's 290-theorem \cite{mB-jH-06}. The theorem is that a nonclassical integral quadratic form over $\mathbb{Q}$ represents all $a \in {\rm BH}$ if and only if it is universal, where
\[
    {\rm BH} = \left\{
            \begin{tabular}{p{3mm}p{3mm}p{3mm}p{3mm}p{3mm}p{3mm}p{3mm}p{3mm}p{3mm}p{3mm}p{3mm}p{3.5mm}p{3.5mm}p{3.5mm}p{3.5mm}}
            1,  & 2,  & 3,  & 5,  & 6,  & 7,  & 10, & 13, & 14, & 15, & 17, & 19,  & 21,  & 22,  & \\ %
            23, & 26, & 29, & 30, & 31, & 34, & 35, & 37, & 42, & 58, & 93, & 110, & 145, & 203, & 290 %
            \end{tabular}
               \right\}.
\]
Therefore $L$ is even universal if and only if $L$ represents all $2a$ for $a \in {\rm BH}$.
\end{Rmk}

\begin{Thm}\label{Thm:EvenUniv}
A primitive even universal binary Hermitian lattice exists over the field $\Q{-m}$ if and only if
$m$ is
\[
    1, 2, 5, 6, 10, 13, 14, 21, 22, 29, 34, 37 \text{ or } 38.
\]
Moreover, we have a complete list of fifty two primitive even universal binary Hermitian lattices
(see Table \ref{tbl:Primitive-even-universal-binary-Hermitian-lattices}).
\begin{table}[h]
\begin{center}
\begin{footnotesize}
\begin{tabular}{p{1.1cm}|lll} \hline%
\multicolumn{1}{c|}{\rm Field} %
            & \multicolumn{3}{c}{\rm Primitive even universal binary Hermitian lattices ($\vol{a}$ means the volume $a\Fro$)} \\ \hline %
$\Q{-1}$    & $[2, {-1+\omega}, 2]_\vol{2}$,              %
            & $[2, 1, 2]_\vol{3}$,                        %
            & $[2,-1+\omega,4]_\vol{6}$,                \\%
            & $[2, 1, 4]_\vol{7}$,                        %
            & $[2, {-1+\omega}, 6]_\vol{10}$,             %
            & $[2, 1, 6]_\vol{11}$,                     \\\hline%

$\Q{-2}$    & $[2, {-1+\omega}, 2]_\vol{1}$,              %
            & $[2, {\omega}, 2]_\vol{2}$,                 %
            & $[2, 1, 2]_\vol{3}$,                      \\%
            & $[2,-1+\omega,4]_\vol{5}$,                  %
            & $[2,\omega,4]_\vol{6}$,                     %
            & $[2,1,4]_\vol{7}$,                        \\%
            & $[2, {-1+\omega}, 6]_\vol{9}$,              %
            & $[2,\omega,6]_\vol{10}$,                    %
            & $[2, 1, 6]_\vol{11}$,                     \\%
            & $[2, {\omega}, 8]_\vol{14}$,                %
            & $[2, {-1+\omega}, 10]_\vol{17}$,            %
            & $[2, 1, 10]_\vol{19}$,                    \\\hline%

$\Q{-5}$    & $[2,2,4\,;-1,\omega,-1]_\vol{1}$,           %
            & $[2,2,4\,;-1,0,-1+\omega]_\vol{1}$,         %
            & $[2,-1+\omega,4]_\vol{2}$,                \\%
            & $[2,2,6\,;0,-1+\omega,-1+\omega]_\vol{2}$,  %
            & $[2,\omega,4]_\vol{6}$,                     %
            &                                           \\\hline%

$\Q{-6}$    & $[2,-1+\omega,4]_\vol{1}$,                  %
            & $[2,2,4\,;-1,0,\omega]_\vol{1}$,            %
            & $[2,\omega,4]_\vol{2}$,                   \\%
            & $[2,2,6\,;0,\omega,\omega]_\vol{2}$,        %
            & $[2,4,10\,;0,\omega,-2+2\omega]_\vol{2}$,   %
            &                                           \\\hline%

$\Q{-10}$   & $[2, -1+\omega, 6]_\vol{1}$,                    %
            & $[2,4,4\,;-1,0,-2+\omega]_\vol{1}$,             %
            & $[2,1,4]_\vol{7}$,                     \\ \hline%

$\Q{-13}$   & $[2,1,4]_\vol{1}$,                              %
            & $[2,4,4\,;0,-1,-1+\omega]_\vol{1}$,             %
            & $[2,4,4\,;0,-1,1+\omega]_\vol{1}$,     \\       %
            & $[2,4,4\,;0,-1,-1+\omega]_\vol{1}$,             %
            & $[2,4,4\,;0,-1,1+\omega]_\vol{1}$,
            &                                        \\ \hline%

$\Q{-14}$   & $[2,4,4\,;-1,0,\omega]_\vol{1}$,                %
            & $[2,\omega,8]_\vol{2}$,                         %
            & $[2,1,4]_\vol{7}$,                     \\ \hline%

$\Q{-17}$   & $[2,4,6\,;-1,-1,-1-\omega]_\vol{1}$,            %
            & $[2,4,6\,;-1,-1,-1+\omega]_\vol{1}$,            %
            & $[2,4,6\,;-1,0,-2-\omega]_\vol{1}$,    \\       %
            & $[2,4,6\,;-1,0,-2+\omega]_\vol{1}$,             %
            & $[2,-1+\omega,10]_\vol{2}$,                     %
            &                                        \\ \hline%

$\Q{-21}$   & $[2,4,6\,;-1,0,\omega]_\vol{1}$,                %
            &                                                 %
            &                                        \\ \hline%

$\Q{-22}$   & $[2,4,6\,;0,-1,\omega]_\vol{1}$,                %
            &                                                 %
            &                                        \\ \hline%

$\Q{-29}$   & $[2,4,8\,;0,-1,1+\omega]_\vol{1}$,               %
            & $[2,4,8\,;0,-1,-1+\omega]_\vol{1}$,              %
            &                                         \\ \hline%

$\Q{-34}$   & $[2,4,10\,;0,-1,-2+\omega]_\vol{1}$,             %
            &                                                  %
            &                                         \\ \hline%

$\Q{-37}$   & $[2,4,10\,;0,-1,1+\omega]_\vol{1}$,              %
            & $[2,4,10\,;0,-1,-1+\omega]_\vol{1}$,             %
            &                                         \\ \hline%

$\Q{-38}$   & $[2,4,10\,;0,-1,\omega]_\vol{1}$.                %
            &                                                  %
            &                                         \\ \hline%

\end{tabular}
\caption{Primitive even universal binary Hermitian lattices} %
\label{tbl:Primitive-even-universal-binary-Hermitian-lattices}%
\end{footnotesize}
\end{center}
\end{table}

\end{Thm}

\begin{proof}
Let $L$ be an even universal binary Hermitian lattice over $\Q{-m}$.
By Lemma \ref{Lem:EvenUniv}, $m \not\equiv 3 \pmod{4}$. To find all candidates of even universal
binary Hermitian lattices, we do escalations and spurious escalations for all possible fields.

\begin{table}[h]
\begin{center}
\begin{footnotesize}
\begin{tabular}{c|l|c|p{5.4cm}|c} \hline%
\multirow{2}*{Fields}  & \multicolumn{1}{c|}{Reduced}             & \multirow{2}*{Truant}&\multicolumn{1}{c|}{Spurious}             & \multirow{2}*{Truant}\\
                       & \multicolumn{1}{c|}{escalation lattices} &                      &\multicolumn{1}{c|}{escalation lattices}  &                      \\\hline %

$\Q{-10}$ & $[2,0,2]_\vol{4}$     & 6     & $[2, -1+\omega, 6]_\vol{1}$                     & none  \\        %
          & $[2,0,4]_\vol{8}$     & 10    & $[2,4,10\,;0,0,2\omega]_\vol{4}$ (not primitive)& none  \\        %
          & $[2,1,4]_\vol{7}$     & 6     & $[2,4,4\,;-1,0,-2+\omega]_\vol{1}$              & none  \\ \hline %

$\Q{-13}$ & $[2,0,2]_\vol{4}$     & 6     & N.A.                                            &      \\        %
          & $[2,0,4]_\vol{8}$     & 10    & $[2,4,4\,;0,-1,\pm 1+\omega]_\vol{1}$           & none \\        %
          & $[2,1,4]_\vol{1}$     & 6     & $[2,4,4\,;0,-1,\pm 1+\omega]_\vol{1}$           & none \\ \hline %

$\Q{-14}$ & $[2,0,2]_\vol{4}$     & 6     & N.A.                                            &      \\        %
          & $[2,0,4]_\vol{8}$     & 10    & $[2,\omega,8]_\vol{2}$                          &      \\        %
          & $[2,1,4]_\vol{7}$     & None  & $[2,4,4\,;-1,0,\omega]_\vol{1}$                 &      \\ \hline 

$\Q{-17}$ & $[2,0,2]_\vol{4}$     & 6     & N.A.                                            &      \\        %
          & $[2,0,4]_\vol{8}$     & 10    & $[2,-1+\omega,10]_\vol{2}$                      & none \\        %
          & $[2,1,4]_\vol{7}$     & 6     & $[2,4,6\,;-1,-1,-1 \pm \omega]_\vol{1}$         & none \\        %
          &                       &       & $[2,4,6\,;-1,0,-2 \pm \omega]_\vol{1}$          & none \\ \hline %

$\Q{-21}$ & $[2,0,2]_\vol{4}$      & 6     & N.A.                                           &       \\        %
          & $[2,0,4]_\vol{8}$      & 10    & $[2,4,6\,;0,-1, \pm 1+\omega]_\vol{1}$         & 14    \\        %
          & $[2,1,4]_\vol{7}$      & 6     & $[2,4,6\,;-1,0,\omega]_\vol{1}$                & none  \\ \hline %

$\Q{-22}$ & $[2,0,2]_\vol{4}$      & 6     & N.A.                                           &       \\        %
          & $[2,0,4]_\vol{8}$      & 10    & $[2,4,6\,;0,-1,\omega]_\vol{1}$                & none  \\        %
          & $[2,1,4]_\vol{7}$      & 6     & N.A.                                           &       \\ \hline %

$\Q{-29}$ & $[2,0,2]_\vol{4}$      & 6     & N.A.                                           &        \\        %
          & $[2,0,4]_\vol{8}$      & 10    & $[2,4,8\,;0,-1, \pm 1+\omega]_\vol{1}$         & none   \\        %
          & $[2,1,4]_\vol{7}$      & 6     & N.A.                                           &        \\ \hline %

$\Q{-34}$ & $[2,0,2]_\vol{4}$      & 6     & N.A.                                           &        \\        %
          & $[2,0,4]_\vol{8}$      & 10    & $[2,4,10\,;0,-1,-2+\omega]_\vol{1}$            & none   \\        %
          & $[2,1,4]_\vol{7}$      & 6     & N.A.                                           &        \\ \hline %

$\Q{-37}$ & $[2,0,2]_\vol{4}$      & 6     & N.A.                                           &        \\        %
          & $[2,0,4]_\vol{8}$      & 10    & $[2,4,10\,;0,-1, \pm 1+\omega]_\vol{1}$        & none   \\        %
          & $[2,1,4]_\vol{7}$      & 6     & N.A.                                           &        \\ \hline %

$\Q{-38}$ & $[2,0,2]_\vol{4}$      & 6     & N.A.                                           &        \\        %
          & $[2,0,4]_\vol{8}$      & 10    & $[2,4,10\,;0,-1,\omega]_\vol{1}$               & none   \\        %
          & $[2,1,4]_\vol{7}$      & 6     & N.A.                                           &        \\ \hline %

\end{tabular}
\caption{Spurious escalation when $m= 10, 13, 14, 17, 21, 22, 29, 34, 37, 38$} %
\label{tbl:2nd m=10,13,14,17,21,22,29,34,37,38}%
\end{footnotesize}
\end{center}
\end{table}

Suppose $m \geq 10$. Since $L$ represents all even positive integers, $L$ contains a lattice $\qf{2}$. Since a unary lattice $\qf{2}$ fails to represent $4$, $L$ contains a lattice
$\binlattice{2}{\alpha}{\conj\alpha}{4}$ for some $\alpha \in \Fro$. From the positive definite condition, $8-{\alpha}{\conj\alpha} \geq 0$, hence $\alpha = 0, \pm1, \pm2$. So we obtain reduced
escalation lattices $\binlattice{2}{0}{0}{2}$, $\binlattice{2}{0}{0}{4}$ and $\binlattice{2}{1}{1}{4}$ which are contained in $L$. Now we will do \emph{spurious escalations} with these reduced binary lattices. Suppose $L$ contains $\ell = \binlattice{2}{0}{0}{2}$. Since $\ell$ fails to represent 6, $L$ contains $\widetilde{\ell} = \terlattice{2}{0}{\beta}{0}{2}{\gamma}{\conj\beta}{\conj\gamma}{6}$ for some $\beta, \gamma \in \Fro$ with $\det \widetilde{\ell} = 0$. We have only one candidate
$\terlattice{2}{0}{-1}{0}{2}{-1+\omega}{-1}{-1+\comega}{6} \cong \binlattice{2}{-1+\omega}{-1+\comega}{6}$ over $\Q{-10}$. Since $\binlattice{2}{0}{0}{4}$ fails to represent $10$ and $\binlattice{2}{1}{1}{4}$ fails to represent $6$, spurious escalation lattices are as following Table \ref{tbl:2nd m=10,13,14,17,21,22,29,34,37,38}. Note that there is no escalation of $\binlattice{2}{0}{0}{4}$, $\binlattice{2}{1}{1}{4}$ over $\Q{-m}$ for all $m \geq 41$ and $m = 26, 30, 33$. All spurious escalation lattices in Table \ref{tbl:2nd m=10,13,14,17,21,22,29,34,37,38} represent all positive even integers except
$\terlattice{2}{0}{-1}{0}{4}{\pm 1+\omega}{-1}{\pm 1+\comega}{6}$ over $\Q{-21}$. Since $\terlattice{2}{0}{-1}{0}{4}{\pm 1+\omega}{-1}{\pm 1+\comega}{6}$ fails to represent $14$ over $\Q{-21}$, we spuriously escalate it with its truant $14$ then we find a unimodular lattice $\terlattice{2}{-1}{0}{-1}{4}{\omega}{0}{\comega}{6}$ which represents all positive even integers. But this lattice also can be spuriously escalated from $\binlattice{2}{1}{1}{4}$ with the truant $6$.

\begin{table}[h]
\begin{center}
\begin{footnotesize}
\begin{tabular}{c|l|c|p{5.4cm}|c} \hline%
\multirow{2}*{Fields}  & \multicolumn{1}{c|}{Reduced}             & \multirow{2}*{Truant}&\multicolumn{1}{c|}{Spurious}             & \multirow{2}*{Truant}\\
                       & \multicolumn{1}{c|}{escalation lattices} &                      &\multicolumn{1}{c|}{escalation lattices}  &                      \\\hline %

$\Q{-1}$ & $[2, 0, 2]_\vol{4}$            & none  & No new lattice                                 &        \\ %
         & $[2, 0, 4]_\vol{8}$            & none  & No new lattice                                 &        \\ %
         & $[2, 0, 6]_\vol{12}$           & none  & $[2,-1+\omega,4]_\vol{6}$                      & none   \\ %
         & $[2, 1, 2]_\vol{3}$            & none  & No new lattice                                 &        \\ %
         & $[2, 1, 4]_\vol{7}$            & none  & N.A.                                           &        \\ %
         & $[2, 1, 6]_\vol{11}$           & none  & N.A.                                           &        \\ %
         & $[2, {-1+\omega}, 2]_\vol{2}$  & none  & No new lattice                                 &        \\ %
         & $[2, {-1+\omega}, 6]_\vol{10}$ & none  & No new lattice                                 &        \\\hline %

$\Q{-2}$ & $[2, 0, 2]_\vol{4}$            & none  & No new lattice                                 &        \\ %
         & $[2, 0, 4]_\vol{8}$            & none  & No new lattice                                 &        \\ %
         & $[2, 0, 6]_\vol{12}$           & none  & $[2,\omega,4]_\vol{6}$                         &        \\ %
         & $[2, 0, 8]_\vol{16}$           & none  & No new lattice                                 &        \\ %
         & $[2, 0, 10]_\vol{20}$          & none  & $[2,-1+\omega,4]_\vol{5}$                      & none   \\ %
         &                                &       & $[2,\omega,6]_\vol{10}$                        & none   \\ %
         & $[2, 1, 2]_\vol{3}$            & none  & No new lattice                                 &        \\ %
         & $[2, 1, 6]_\vol{11}$           & none  & No new lattice                                 &        \\ %
         & $[2, 1, 10]_\vol{91}$          & none  & No new lattice                                 &        \\ %
         & $[2, {\omega}, 2]_\vol{2}$     & none  & No new lattice                                 &        \\ %
         & $[2, {\omega}, 8]_\vol{14}$    & none  & $[2,1,4]_\vol{7}$                              & none   \\ %
         & $[2, {\omega}, 10]_\vol{18}$   & 14    & $[2,\omega,4]_\vol{6}$                         & none   \\ %
         & $[2, {-1+\omega}, 2]_\vol{1}$  & none  & No new lattice                                 &        \\ %
         & $[2, {-1+\omega}, 6]_\vol{9}$  & none  & No new lattice                                 &        \\ %
         & $[2, {-1+\omega}, 10]_\vol{17}$& none  & No new lattice                                 &        \\  \hline %

$\Q{-5}$ & $[2,0,2]_\vol{4}$              & 6     & $[2,2,6\,;0,-1+\omega,-1+\omega]_\vol{2}$      & none   \\ %
         & $[2,0,4]_\vol{8}$              & none  & $[2,2,4\,;-1,0,-1+\omega]_\vol{1}$             &        \\ %
         &                                &       & $[2,2,4\,;-1,\omega,-1]_\vol{1}$               &        \\ %
         &                                &       & $[2,4,6\,:0,0,-2+2\omega]_\vol{4}$ (not primitive)         
                                                                                                   &        \\ 
         & $[2,1,4]_\vol{7}$              & 6     & $[2,2,4\,;-1,0,-1+\omega]_\vol{1}$             & none   \\ 
         &                                &       & $[2,2,4\,;-1,\omega,-1]_\vol{1}$               & none   \\ 
         & $[2,\omega,4]_\vol{6}$         & none  & $[2,2,4\,;-1,0,-1+\omega]_\vol{1}$             & none   \\ %
         &                                &       & $[2,2,4\,;-1,\omega,-1]_\vol{1}$               & none   \\ %
         & $[2,-1+\omega,4]_\vol{2}$      & none  & $[2,2,4\,;-1,0,-1+\omega]_\vol{1}$             & none   \\ %
         &                                &       & $[2,2,4\,;-1,\omega,-1]_\vol{1}$               & none   \\ \hline %

$\Q{-6}$ & $[2,0,2]_\vol{4}$              & 6     & $[2,2,6\,;0,\omega,\omega]_\vol{2}$            & none   \\ %
         & $[2,0,4]_\vol{8}$              & 10    & $[2,2,4\,;-1,0,\omega]_\vol{1}$                & none   \\ %
         &                                &       & $[2,4,6\,;0,0,2\omega]_\vol{4}$                & none   \\ %
         &                                &       & $[2,4,10\,;0,\omega,-2+2\omega]_\vol{2}$       & none   \\ %
         & $[2,1,4]_\vol{7}$              & 6     & No New Lattice                                 & none   \\ %
         & $[2,\omega,4]_\vol{2}$         & none  & $[2,2,4\,;-1,0,\omega]_\vol{1}$                & none   \\ %
         & $[2,-1+\omega,4]_\vol{1}$      & none  & No new lattice                                 &        \\ \hline %

\end{tabular}
\caption{Spurious escalation when $m=1, 2, 5, 6$} %
\label{tbl:2nd m=1,2,5,6}%
\end{footnotesize}
\end{center}
\end{table}

Suppose $m=1$. Since $L$ contains a lattice $\qf{2}$ and the unary lattice $\qf{2}$ fails to represent $6$ over $\Q{-1}$, $L$ contains a lattice $\binlattice{2}{\alpha}{\conj\alpha}{6}$, for some $\alpha \in \Fro$. From the positive definite condition, $12-{\alpha}{\conj\alpha} \geq 0$. So we obtain reduced lattices $\binlattice{2}{0}{0}{2a}$, $\binlattice{2}{1}{1}{2a}$, $\binlattice{2}{-1+\omega}{-1+\comega}{2}$ and $\binlattice{2}{-1+\omega}{-1+\comega}{6}$, where $a=1$, $2$, $3$. Even though these lattices are even universal, we should do spurious escalation with these lattices to find all binary even universal lattices. For $\ell = \binlattice{2}{0}{0}{6}$, a spuriously escalated lattice of $\ell$ has a form ${\widetilde \ell} =
\terlattice{2}{0}{\alpha}{0}{6}{\beta}{\conj\alpha}{\conj\beta}{2b}$, where $\alpha = \alpha_1 + \omega \alpha_2$, $\beta = \beta_1 + \omega \beta_2 \in \Fro$ such that $\mid \alpha_i \mid \leq 1$, $\mid \beta_i \mid \leq 3$ and $b \in \N$. Since any sublattice of ${\widetilde \ell}$ is positive definite, $b$ is bounded. Moreover ${\widetilde \ell}$ satisfies $\det{\widetilde \ell} = 0$ and the volume ideal of any sublattice of ${\widetilde\ell}$ is contained in the volume ideal $\Frv \ell = 12\Fro$. So we get spurious escalation lattices $\binlattice{2}{1}{1}{2}$ and $\terlattice{2}{-1+\omega}{0}{-1+\comega}{4}{-3+3\omega}{0}{-3+3\comega}{6} \cong \binlattice{2}{-1+\omega}{-1+\comega}{4}$. Note that $\binlattice{2}{1}{1}{2}$ was already found at the previous step.

With almost identical arguments, we get the following Table \ref{tbl:2nd m=1,2,5,6}. Note that
$\Q{-1}$ and $\Q{-2}$ are P.I.D.

Therefore, we have candidates of even universal Hermitian lattices over $\Q{-m}$ as listed in Table \ref{tbl:2nd m=10,13,14,17,21,22,29,34,37,38} and Table \ref{tbl:2nd m=1,2,5,6}. The proof of the even universalities of these lattices proceeds by Bhargava and Hanke's 290-theorem(See Remark \ref{Rmk:BH}).
\end{proof}

\section{finiteness theorems}

In this section, we will find a set $C[m]$ of \emph{even critical numbers} for each imaginary quadratic field $\Q{-m}$ which admits \emph{even universal binary Hermitian lattices}. From the Bhargava and Hanke's 290-theorem, $L$ is even universal if and only if $L$ represents all $2a$ for $a \in {\rm BH}$ (see Remark \ref{Rmk:BH}). That implies $C[m] \subseteq {\rm BH}^*$ where ${\rm BH}^* = \{2a \mid a \in {\rm BH} \}$. But these twenty nine positive integers in the set ${\rm BH}$ are redundant for some imaginary quadratic fields. So we will seek an optimal set $C[m]$ of even critical numbers for each imaginary quadratic field $\Q{-m}$ which admit even universal binary Hermitian lattices.

The following Lemma \ref{Lem:only-one-exception} says that an even truant of a Hermitian lattice is an even critical number, that is essential to the proof of the first part of Theorem \ref{Thm:Main-Theorem}.

\begin{Lem}\label{Lem:only-one-exception}
Let $L$ be an even Hermitian lattice over $\Q{-m}$. If $L$ represents all even positive integers
less than $2a$ but does not represent $2a$, then
\[
L' = L \perp \qf{2a+2,2a+2,2a+2,2a+2} \perp \qf{4a+2}
\]
represents all even positive integers except $2a$.
\end{Lem}

\begin{proof}
By Lagrange's four square theorem, it is clear.
\end{proof}

\begin{Thm}\label{Thm:Main-Theorem}
If a primitive even Hermitian lattice over $\Q{-m}$ represents the following even critical numbers
in $C[m]$ (see Table \ref{tbl:Even-Critical-Number}), then it is even universal.
\begin{table}[h]
\begin{center}
\begin{footnotesize}
\begin{tabular}{l|l}  \hline
\multicolumn{1}{c|}{\rm Field}& \multicolumn{1}{c}{\rm The Set $C[m]$ of Even critical numbers }  \\ \hline %
$\Q{-1}$                  & $2$, $6$,\\%
$\Q{-2}$                  & $2$, $10$, $14$, \\%
$\Q{-5}$                  & $2$, $4$, $6$, \\%
$\Q{-6}$                  & $2$, $4$, $6$, $10$, \\%
$\Q{-10}$                 & $2$, $4$, $6$, $10$, $12$, $14$, $30$,\\%
$\Q{-13}$                 & $2$, $4$, $6$, $10$, $12$, $14$, $20$,\\%
$\Q{-14}$                 & $2$, $4$, $6$, $10$, $12$, $14$, $20$, $26$, $42$,\\%
$\Q{-17}$                 & $2$, $4$, $6$, $10$, $12$, $14$, $20$, $26$, $28$, $30$,\\%
$\Q{-21}$                 & $2$, $4$, $6$, $10$, $12$, $14$, $20$, $26$, $28$, $30$, $34$, $38$,\\%
$\Q{-22}$                 & $2$, $4$, $6$, $10$, $12$, $14$, $20$, $26$, $28$, $30$, $34$, $38$, $42$, \\%
$\Q{-26}$                 & $2$, $4$, $6$, $10$, $12$, $14$, $20$, $26$, $28$, $30$, $34$, $38$, $42$, $46$, \\%
$\Q{-29}$                 & $2$, $4$, $6$, $10$, $12$, $14$, $20$, $26$, $28$, $30$, $34$, $38$, $42$, $46$, $52$, \\%
$\Q{-30}$                 & $2$, $4$, $6$, $10$, $12$, $14$, $20$, $26$, $28$, $30$, $34$, $38$, $42$, $46$, $52$, $58$, \\%
$\Q{-33}$                 & $2$, $4$, $6$, $10$, $12$, $14$, $20$, $26$, $28$, $30$, $34$, $38$, $42$, $46$, $52$, $58$, $60$, $62$, \\%
$\Q{-34}$                 & $2$, $4$, $6$, $10$, $12$, $14$, $20$, $26$, $28$, $30$, $34$, $38$, $42$, $46$, $52$, $58$, $60$, $62$, \\%
$\Q{-37}$                 & $2$, $4$, $6$, $10$, $12$, $14$, $20$, $26$, $28$, $30$, $34$, $38$, $42$, $46$, $52$, $58$, $60$, $62$, $68$, $70$,\\%
$\Q{-38}$                 & $2$, $4$, $6$, $10$, $12$, $14$, $20$, $26$, $28$, $30$, $34$, $38$, $42$, $46$, $52$, $58$, $60$, $62$, $68$, $70$, $74$\\ \hline%
\end{tabular}
\caption{The set $C[m]$ of even critical numbers for $\Q{-m}$} %
\label{tbl:Even-Critical-Number}%
\end{footnotesize}
\end{center}
\end{table}
\end{Thm}

\begin{proof}
By Lemma \ref{Lem:only-one-exception}, if $2a$ is a truant of a lattice $L$ over $\Q{-m}$, then $2a$ is a critical number of $\Q{-m}$. Most of all process of finding a set of critical numbers
$C[m]$ is based on the proof of Theorem \ref{Thm:EvenUniv} which indicates the truants or exceptional numbers of $L$.

If $m=1$, the eight escalated lattices (not necessarily primitive) with truants $2$ and $6$ over $\Q{-1}$ are all even universal binary Hermitian lattices. So any lattice which represents both numbers $2$ and $6$ contains an even universal binary Hermitian lattice as a sublattice. Therefore
    \[
    C[1] = \{ 2, 6 \}.
    \]

If $m=2$, the escalations by the truants $2$ and $10$ give $14$ even lattices (which are not necessarily primitive) over $\Q{-2}$. These binary even lattices are all even universal except $[2,\omega,10]$. The lattice $[2,\omega,10]$ represents all positive even integers except $14$ and escalated lattices of $[2,\omega,10]$ by the truant $14$ are even universal including the binary lattice $[2,\omega,4]$. Therefore
    \[
    C[2] = \{ 2, 10, 14 \}.
    \]

With the identical arguments, we can show that
    \[
    C[5] = \{ 2, 4, 6 \}  \quad \text{ and } \quad  C[6] = \{ 2, 4, 6, 10 \}.
    \]

From now on, suppose $m=10$, $13$, $\cdots$, $38$. We begin with listing lattices with their truant that is essential member of $C[m]$.

The escalations by truants $2$ and $4$ give lattices $\binlattice{2}{0}{0}{2}$,
$\binlattice{2}{0}{0}{4}$ and $\binlattice{2}{1}{1}{4}$, whose truants are $6$, $10$ and $6$ respectively. So $2, 4, 6, 10 \in C[m]$.

If $m \geq 10$, then the truants of $\terlattice{2}{1}{1}{1}{4}{-1}{1}{-1}{6}$ and
$\terlattice{2}{0}{1}{0}{4}{1}{1}{1}{6}$ are $12$ and $14$, respectively. So $12, 14 \in C[m]$.

If $m=10$, then the truant of $\terlattice{2}{0}{\omega}{0}{4}{0}{\comega}{0}{10}$ is $30$. So $30
\in C[10]$.

If $m \geq 13$, then the truant of $\terlattice{2}{0}{0}{0}{4}{1}{0}{1}{4}$ is $20$. So $20 \in C[m]$. %

If $m \geq 14$, then the truant of $\terlattice{2}{0}{0}{0}{4}{1}{0}{1}{10}$ is $26$. So $26 \in C[m]$. %

If $m=14$, then the truant of
${\begin{small}{\begin{pmatrix}%
  2       & 0  & \omega  & 0  \\
  0       & 4  & -2      & -2 \\
  \comega & -2 & 8       & 1  \\
  0       & -2 & 1       & 8  \\
\end{pmatrix}}\end{small}}$ is $42$. So $42 \in C[14]$.

If $m \geq 17$, then the truants of $\terlattice{2}{0}{1}{0}{4}{2}{1}{2}{6}$ and $\terlattice{2}{1}{0}{1}{4}{2}{0}{2}{4} \perp \qf{10}$ are $28$ and $30$, respectively. So $28,30\in C[m]$.

If $m \geq 21$, then the truants of $\terlattice{2}{1}{1}{1}{4}{-1}{1}{-1}{6}$ and
$\terlattice{2}{1}{1}{1}{4}{0}{1}{0}{6} \perp \qf{18}$ are $34$ and $38$, respectively. So $34, 38 \in C[m]$.

If $m \geq 22$, then the truant of $\terlattice{2}{0}{0}{0}{2}{1}{0}{1}{4}$ is $42$. So $42 \in C[m]$. %

If $m \geq 26$, then the truant of $\terlattice{2}{0}{0}{0}{4}{1}{0}{1}{6}$ is $46$. So $46 \in C[m]$. %

If $m \geq 29$, then the truant of $\qf{2} \perp \terlattice{4}{0}{2}{0}{6}{3}{2}{3}{18}$ is $52$. So $52 \in C[m]$. %

If $m \geq 30$, then the truant of $\terlattice{2}{0}{1}{0}{4}{1}{1}{1}{10}$ is $58$. So $58 \in C[m]$. %

If $m \geq 33$, then the truants of
${\begin{small}{\begin{pmatrix}%
  2 & 0 & 1  & 1 \\
  0 & 4 & 0  & 0 \\
  1 & 0 & 10 & 0 \\
  1 & 0 & 0  & 58\\
\end{pmatrix}}\end{small}}$ and $\terlattice{2}{0}{0}{0}{4}{1}{0}{1}{8}$ are $60$ and $62$, respectively. So  $60, 62 \in C[m]$. %

If $m \geq 37$, then the truants of
$\terlattice{2}{1}{0}{1}{4}{2}{0}{2}{6} \perp \qf{34}$ and $\binlattice{2}{1}{1}{4} \perp \qf{2,28}$ are $68$ and $70$, respectively. So $68, 70 \in C[m]$.

If $m= 38$, then the truant of
${\begin{small}{\begin{pmatrix}%
  2 & 0 & 1  & 0 \\
  0 & 4 & 1  & 1 \\
  1 & 1 & 10 & 0 \\
  0 & 1 & 0  & 58\\
\end{pmatrix}}\end{small}}$ is $74$. So $74 \in C[38]$.

We showed that the numbers given in Table \ref{tbl:Even-Critical-Number} are elements of $C[m]$. Now we prove that $C[m]$ has no more numbers.

For $m=10$, let a set $D[10] = \{ 2, 4, 6, 10, 12, 14, 30 \}$. It suffices to show that if $L$ represents all numbers in $D[10]$, $L$ represents all numbers in ${\rm BH}^*$. By escalation we can find a finite set
\[
S = \{ M \mid \text{ binary or ternary even Hermitian lattice such that } 2, 4, 6, 10 \ra M \}.
\]
In the escalation process, $L$ should contain some $M \in S$ as a sublattice. By computer calculation, we confirmed that all $M \in S$ represents all numbers $n \in {\rm BH}^{*} \setminus D[10]$. So $D[10] \supseteq C[10]$. Similarly to $C[10]$, we can check that the Table \ref{tbl:Even-Critical-Number} is optimal.
\end{proof}



\begin{thebibliography}{12}
\bibitem{mB-00}
Bhargava, M.: On the Conway-Schneeberger fifteen theorem, Contemp. Math. {\bf 272} (2000), 27--37.

\bibitem{mB-jH-06}
Bhargava, M., Hanke, J.: \url{http://www.math.duke.edu/~jonhanke/290/Universal-290.html}.

\bibitem{jhC-00}
Conway, J. H.: Universal quadratic forms and the fifteen theorem, Contemp. Math. {\bf 272} (2000), 23--26.

\bibitem{agE-aK-97(1)}
Earnest, A. G., Khosravani, A.: Universal binary Hermitian forms, Math. Comp. {\bf 66} (1997), 1161--1168.

\bibitem{ljG-78}
Gerstein, L. J.: Classes of definite Hermitian forms, Amer. J. Math. {\bf 100} (1978), 81--97.

\bibitem{hI-00}
Iwabuchi, H.: Universal binary positive definite Hermitian lattices, Rocky Mountain J. Math.
{\bf 30} (2000), 951--959.

\bibitem{nJ-40}
Jacobson, N.: A note on hermitian forms, Amer. Math. Soc. {\bf 46} (1940), 264--268.

\bibitem{rJ-62}
Jacobowitz, R.: Hermitian forms over local fields, Amer. J. Math. {\bf 84} (1962), 441--465.

\bibitem{bmK-jyK-psP(1)}
Kim, B. M., Kim, J.Y., Park, P.-S.: The fifteen theorem for universal Hermitian lattices over
imaginary quadratic fields, arXiv:0710.4991v2.

\bibitem{bmK-jyK-psP(2)}
Kim, B. M., Kim, J.Y., Park, P.-S.: Binary normal regular Hermitian lattices over imaginary
quadratic fields, arXiv:0710.4998v2.

\bibitem{jhK-psP-07}
Kim, J.-H., Park, P.-S.: A few uncaught universal Hermitian forms, Proc. Amer. Math. Soc.
{\bf 135} (2007), 47--49.

\bibitem{otO-73}
O'Meara, O. T.: Introduction to Quadratic Forms, Spinger-Verlag, New York, 1973.

\end{thebibliography}
\end{document}